\documentclass{amsart}
\usepackage{amssymb}
\usepackage{mathrsfs}
\usepackage{hyperref}
\usepackage[curve,arrow,matrix]{xy}

\theoremstyle{plain}
\newtheorem{thm}{Theorem}[section]
\newtheorem{prop}[thm]{Proposition}
\newtheorem{lemma}[thm]{Lemma}

\newcommand{\deqno}{\refstepcounter{thm}(\thethm)}
\theoremstyle{definition}
\newtheorem{defn}[thm]{Definition}
\newtheorem{disc}[thm]{Discussion}
\newtheorem{note}[thm]{Notation}

\newtheorem{exam}[thm]{Example}
\newtheorem{axioms}[thm]{Axioms}

\theoremstyle{remark}

\newcounter{item}

\newenvironment{alist}[1][1]{\begin{list}
  {\textup{(\alph{item})}}{\usecounter{item} \setcounter{item}{#1}\addtocounter{item}{-1}
  \setlength{\itemsep}{0ex}
  \setlength{\topsep}{0ex} \setlength{\parsep}{0ex} \setlength{\labelwidth}{15mm}
  \setlength{\leftmargin}{10mm} } }{\end{list}}
\newenvironment{nlist}[1][1]{\begin{list}
  {\textup{(\arabic{item})}}{\usecounter{item} \setcounter{item}{#1}\addtocounter{item}{-1}
  \setlength{\itemsep}{0ex}
  \setlength{\topsep}{0ex} \setlength{\parsep}{0ex} \setlength{\labelwidth}{15mm}
  \setlength{\leftmargin}{10mm} } }{\end{list}}

\newcommand{\ba}{{\bf a}}

\newcommand{\bx}{{\bf x}}
\newcommand{\by}{{\bf y}}

\newcommand{\bv}{{\bf v}}

\newcommand{\Fe}{{\bf F}^e}

\newcommand{\tensor}{\otimes}
\newcommand{\Hom}{\textup{Hom}}
\newcommand{\Ext}{\textup{Ext}}

\newcommand{\im}{\textup{Im}}
\newcommand{\id}{\textup{id}}
\newcommand{\tr}{\textup{tr}}

\newcommand{\ra}{\rightarrow}

\newcommand{\isom}{\cong}

\newcommand{\wh}{\widehat}

\begin{document}

\title[Closure Operations That Induce Big C-M Modules]
{A Characterization of Closure Operations \\ That 
Induce Big Cohen-Macaulay Modules}

\author{Geoffrey D. Dietz}
\address{Department of Mathematics,
Gannon University, Erie, PA 16541}
\email{gdietz@member.ams.org}


\date{November 3, 2010}

\begin{abstract}
The intent of this paper is to present a set of axioms
that are sufficient for a closure operation 
to generate a balanced big Cohen-Macaulay module $B$ over a complete
local domain $R$. Conversely, we show that if such a $B$ exists over $R$,
then there exists a closure operation that satisfies the given axioms.
\end{abstract}

\maketitle

In equal characteristic,
the tight closure operation has been used to present proofs
of the existence of balanced big Cohen-Macaulay modules and
algebras. (See \cite{D}, \cite{Ho75}, \cite{Ho94}, and \cite{HH94} 
for example.)
Finding a closure operation in mixed characteristic that is 
powerful enough to produce big Cohen-Macaulay modules or algebras
has been an elusive yet very important goal for some time now, as 
the existence
of such a closure operation would yield major results
related to the homological conjectures in commutative algebra.

In this article, we present a list of seven axioms for a closure
operation defined for finitely generated modules over a complete
local domain $R$. After deriving some simpler consequences of
the axioms (including colon-capturing), 
we prove that a closure operation satisfying the axioms
implies the existence of a \textit{balanced big Cohen-Macaulay module}
over $R$. A balanced big Cohen-Macaulay module $B$ over $(R,\mathfrak{m})$ 
is an
$R$-module where every system of parameters in $R$ is
a regular sequence on $B$ (\emph{i.e.}, if $x_1,\ldots, x_{k+1}$ forms part of
a system of parameters in $R$ and $b\in (x_1,\ldots,x_k):_B x_{k+1}$,
then $b\in (x_1,\ldots,x_k)B$) and $\mathfrak{m}B\neq B$.
Our main tools for proving this theorem will be Hochster's
method of modifications \cite{Ho75} and the use of analogues of 
phantom extensions developed by Hochster and Huneke \cite{HH94}.

We then show that the existence of a
balanced big Cohen-Macaulay module over a complete local domain
can be used to create a closure operation that satisfies all
of the axioms. 

Finally, we demonstrate that all of the axioms are satisfied
by tight closure. We also examine the axioms in relation
to several other common closure operations associated with tight 
closure theory.

\section{The Closure Axioms and Consequences}

Let $R$ be a complete local domain, and let $N\subseteq M$ be
finitely generated modules.
An operation satisfying Axioms~(1)--(5) below will be 
called a \textit{closure operation} and will be
denoted by $N_M^{\natural}$ 
for the closure of $N$ 
within $M$. Throughout the paper we will call this the
\textit{$\natural$-closure} of $N$ in $M$. In Section~5,
we will show that tight closure, plus closure, Frobenius closure,
and solid closure all satisfy
Axioms~(1)--(6), while tight, plus, and solid closures 
(at least in positive characteristic) also satisfy
Axiom~(7).

\begin{axioms} \label{theaxioms}
Let $(R, \mathfrak{m})$ be a fixed complete local domain. Let
$I$ be an arbitrary ideal of $R$, 
and let $N$, $M$, and $W$ be arbitrary finitely generated
$R$-modules with $N\subseteq M$.
\begin{nlist}
  
  \item $N_M^\natural$ is a submodule of $M$ containing $N$. 
  
  \item $(N_M^\natural)_M^\natural = N_M^\natural$; \emph{i.e.}, the 
  $\natural$-closure of $N$ in $M$ is closed in $M$. 

  \item If $N\subseteq M\subseteq W$, 
  then $N_W^\natural \subseteq M_W^\natural$. 
  
  \item Let $f:M\ra W$ be a homomorphism. 
  Then $f(N_M^\natural)\subseteq f(N)_W^\natural$.
  
  \item If $N_{M}^{\natural}=N$, then $0_{M/N}^{\natural} = 0$.

  \item The ideals $\mathfrak{m}$ and $0$ are $\natural$-closed; \emph{i.e.},
  $\mathfrak{m}^\natural_R = \mathfrak{m}$ and $0^\natural_R = 0$.
 
  \item Let $x_1,\ldots, x_{k+1}$ be a partial system of parameters for
  $R$, and let $J=(x_1,\ldots,x_k)$. Suppose that there exists a surjective
  homomorphism $f:M\to R/J$ and $v\in M$ such that $f(v) = x_{k+1} + J$. Then
  $$
  (Rv)^\natural_M \cap \textrm{ker}\, f \subseteq (Jv)^\natural_M.
  $$

\end{nlist} \end{axioms}

The final axiom will be referred to as the \textit{generalized colon-capturing
property} due to the fact
that we will derive ordinary colon-capturing from it.

The following is a list of basic properties of a closure operation
that mimics part of the list given 
for tight closure of modules in \cite[Section 8]{HH90}. 

\begin{lemma} \label{proplist} Let $R$ be a complete local domain
possessing a closure operation that satisfies Axioms~\textnormal{(1)--(5)}.
In the following, $N$, $N'$, and $N_{i}\subseteq M_{i}$ are all 
$R$-submodules of the finitely generated $R$-module $M$. 
\begin{alist}

\item Let $N'\subseteq N\subseteq M$. Then $u\in N_{M}^{\natural}$ if and 
only if $u+N'\in (N/N')_{M/N'}^{\natural}$.

\item If $\mathcal{I}$ is a finite set,
$N=\bigoplus_{i\in\mathcal{I}}N_{i}$, and $M=\bigoplus_{i\in\mathcal{I}} 
M_{i}$, then we have $N_{M}^{\natural} = \bigoplus_{i\in\mathcal{I}} 
(N_{i})_{M_i}^{\natural}$. 

\item Let $\mathcal{I}$ be any set. If $N_{i}\subseteq M$ for all 
$i\in\mathcal{I}$, then $(\bigcap_{i\in\mathcal{I}} 
N_{i})_{M}^{\natural}\subseteq \bigcap_{i\in\mathcal{I}} 
(N_{i})_{M}^{\natural}$.

\item Let $\mathcal{I}$ be any set. If $N_{i}$ is $\natural$-closed
in $M$ for all $i\in\mathcal{I}$, then $\bigcap_{i\in\mathcal{I}} N_{i}$ 
is $\natural$-closed in $M$.

\item $(N_{1}+N_{2})_{M}^{\natural} = ((N_{1})_{M}^{\natural} + 
(N_{2})_{M}^{\natural})_{M}^{\natural}$. 
\end{alist}\end{lemma}
~
\begin{proof}
~
\begin{alist} 
\item Since $(M/N')/(N/N')\isom M/N$, it is enough to prove the
assertion for the case 
where $N'=0$. Let
$\pi:M\ra M/N$ be the natural surjection. If $u\in N_M^\natural$, then
Axiom~(4) implies that $u+N=\pi(u)\in (\pi(N))_{M/N}^\natural = 
0_{M/N}^\natural$. Conversely, suppose that $u+N\in 0_{M/N}^\natural$. 
By Axiom~(1), there is a natural homomorphism $M/N\ra M/N_M^\natural$.
Applying Axiom~(4) to this map yields $u+N_M^\natural\in
0_{M/N_M^\natural}^\natural$. By Axiom~(2) and Axiom~(5), 
$0$ is $\natural$-closed in
$M/N_M^\natural$, so $u+N_M^\natural$ is zero in the quotient $M/N_M^\natural$.

\item Let $\pi_i:M\ra M_i$ be the natural projection, and let
$\iota_i:M_i\ra M$ be the natural inclusion for all $i$. By Axiom~(4),
$$
\pi_i(N_M^\natural)\subseteq \pi_i(N)_{M_i}^\natural = (N_i)_{M_i}^\natural.
$$
Therefore, $N_M^\natural\subseteq \bigoplus_i
(N_i)_{M_i}^\natural$ in $M=\bigoplus_i M_i$. Conversely, we apply
Axiom~(4) to the map $\iota_i$ to see that $\iota_i((N_i)_{M_i}^\natural)
\subseteq \iota_i(N_i)_M^\natural$, which is contained in
$N_M^\natural$ by Axiom~(3). Thus, $\bigoplus_i (N_i)_{M_i}^\natural\subseteq
N_M^\natural$.

\item Let $u\in (\bigcap_{i\in\mathcal{I}} 
N_{i})_{M}^{\natural}$. By Axiom~(3), $u\in (N_{i})_{M}^{\natural}$
for all $i\in\mathcal{I}$.

\item Since each $N_{i}$ is $\natural$-closed, the conclusion follows 
from (c) and Axiom~(1). 

\item As $N_{i}\subseteq (N_{i})_{M}^{\natural}$ by Axiom~(1), 
Axiom~(3) implies that
$$(N_{1}+N_{2})_{M}^{\natural}\subseteq ((N_{1})_{M}^{\natural} + 
(N_{2})_{M}^{\natural})_{M}^{\natural}.$$
Conversely, $N_{i}\subseteq 
N_{1}+N_{2}$, with Axiom~(3), implies that 
$(N_{i})_{M}^{\natural}\subseteq (N_{1}+N_{2})_{M}^{\natural}$. 
Therefore, 
$(N_{1})_{M}^{\natural} + (N_{2})_{M}^{\natural} \subseteq
(N_{1}+N_{2})_{M}^{\natural}$. 
Axiom~(2) then yields
$$((N_{1})_{M}^{\natural} + 
(N_{2})_{M}^{\natural})_{M}^{\natural}\subseteq 
((N_{1}+N_{2})_{M}^{\natural})_{M}^{\natural} = 
(N_{1}+N_{2})_{M}^{\natural}.$$ 
\end{alist}~\end{proof}

Although we only need Axiom~(7)
in the form given in order to produce a 
balanced big Cohen-Macaulay module, we can derive a more general
property that does not require the map $M\to R/J$ to be surjective. 
In fact, if an operation satisfies Axioms~(1), (3), and (4), then
Axiom~(7) is equivalent to Lemma \ref{genaxiom7}. 

\begin{lemma}\label{genaxiom7}
Let $R$ be a complete local domain possessing an operation
that satisfies Axioms~(1), (3), (4), and (7).
Let $x_1,\ldots, x_{k+1}$ be a partial system of parameters for
$R$, and let $J=(x_1,\ldots,x_k)$. Suppose that there exists a 
homomorphism $f:M\to R/J$ such that $f(v) = x_{k+1} + J$. Then
$$
(Rv)^\natural_M \cap \textnormal{ker}\, f \subseteq (Jv)^\natural_M.
$$
\end{lemma}
\begin{proof}
Given the setup above, suppose that 
$u\in (Rv)^\natural_M \cap \textrm{ker}\, f$. Define a new homomorphism
$\phi: M\oplus R \to R/J$ by $\phi(m,r) = f(m) + r + J$. The map $\phi$
is surjective as $\phi(0,1) = 1 + J$. Also, $\phi(v,0) = f(v) = x_{k+1}+J$.
Therefore, $\phi$ satisfies the assumptions of Axiom~(7), and we may
conclude that 
$$
(R(v,0))^\natural_{M\oplus R} \cap \textrm{ker}\, \phi \subseteq
(J(v,0))^\natural_{M\oplus R}.
$$
Lemma \ref{proplist}(b) implies that $(R(v,0))^\natural_{M\oplus R}
= (Rv)^\natural_M \oplus 0^\natural_R$ and that 
$(J(v,0))^\natural_{M\oplus R}
= (Jv)^\natural_M \oplus 0^\natural_R$. Thus, 
$(u,0) \in (Jv)^\natural_M \oplus 0^\natural_R$, and so $u\in (Jv)^\natural_M$
as desired.
\end{proof}

We can now derive a colon-capturing property analogous to one found in
tight closure theory.

\begin{prop}[colon-capturing]
Let $R$ be as in the previous lemma.
Let $x_1,\ldots,x_{k+1}$ be a partial system of parameters in $R$. Then
$(x_1,\ldots,x_k): x_{k+1} \subseteq (x_1,\ldots,x_k)^\natural_R.$
\end{prop}
\begin{proof}
Apply the hypotheses above to Lemma \ref{genaxiom7} with $M=R$,
$J=(x_1,\ldots,x_k)$, $f:R\to R/J$ given by $f(r) = rx_{k+1} + J$,
and $v=1$. Then $f(1) = x_{k+1}+J$, $(Rv)^\natural_M = R$ by Axiom~(1), 
$\textrm{ker}\, f = J: x_{k+1}$, and $(Jv)^\natural_M = J^\natural_R$.
\end{proof}

\section{$\natural$-Phantom Extensions}

In this section we define a notion of phantom extensions
for a closure operation satisfying Axioms~(1)--(5).
Phantom extensions and module modifications were used in
\cite{HH94} to produce a new proof of the existence of big
Cohen-Macaulay modules in positive characteristic. Our study
of $\natural$-phantom extensions will lead to the existence of 
a balanced big Cohen-Macaulay module over $R$.

In \cite[Section 5]{HH94}, a map
$\alpha:N\ra M$ of finitely generated $R$-modules 
is a \textit{phantom extension} if there exists $c\in R$ but
not in any minimal prime 
such that for all $e\gg 0$,
there exists $\gamma_e:\Fe(M)\ra\Fe(N)$ such that $\gamma_e\circ \Fe(\alpha)
= c(\id_{\Fe(N)})$, where $\Fe$ is the iterated Frobenius functor. 

Via the Yoneda correspondence, every short exact
sequence 
$$
0\ra N\overset\alpha\ra M\ra Q\ra 0
$$
corresponds to a unique element
$\epsilon$ of $\Ext_R^1(Q,N)$. Let $P_\bullet$ be a projective resolution 
of $Q= M/\alpha(N)$. Then $\Ext_R^1(Q,N)$ is isomorphic to 
$H^1(\Hom_R(P_\bullet,N))$.

Given the map $\alpha:N\ra M$ and the corresponding
element $\epsilon$ of $\Ext_R^1(Q,N)$, Hochster and Huneke
called $\epsilon$ a \textit{phantom element}
if a cocycle representative of $\epsilon$ in $\Hom_R(P_1,N)$ is in the tight 
closure of 
$$
\im(\Hom_R(P_0,N)\ra \Hom_R(P_1,N))
$$ 
within $\Hom_R(P_1,N)$. In the 
case where $N=R$, Hochster and Huneke provide the following equivalence.

\begin{thm}[Theorem 5.13, \cite{HH94}] Let $R$ be a reduced
Noetherian ring of positive characteristic. An exact
sequence 
$$0\ra R\overset\alpha\ra M\ra Q\ra 0$$
is a phantom extension
if and only if the corresponding element $\epsilon$ in $\Ext_R^1(Q,R)$
is phantom in the sense described above. \end{thm}

We will use this latter property to define \textit{$\natural$-phantom 
extensions} with respect to $\natural$-closure.

\begin{defn} Let $R$ be a complete local domain
possessing a closure operation satisfying Axioms~(1)--(5).
Let $M$ be a finitely generated $R$-module 
and $\alpha:R\ra M$ an injective $R$-linear map. With $Q=M/\alpha(R)$, 
we have an induced short exact sequence 
$$0\ra R\overset\alpha\ra M\ra Q\ra 0.$$
Let $\epsilon\in\Ext_R^1(Q,R)$ be the element corresponding to this
short exact sequence via the Yoneda correspondence. 
Use $(-)^\vee$ to denote the operation $\Hom_R(-,R)$.
If $P_\bullet$ is
a projective resolution of $Q$ consisting of finitely generated projective
modules $P_i$, then we will say that $\epsilon$ is 
\textit{$\natural$-phantom} if a cocycle representing 
$\epsilon$ in $P_1^\vee$ is in
$\im(P_0^\vee\ra P_1^\vee)^\natural_{P_1^\vee}$. We will
call $\alpha$ a \textit{$\natural$-phantom extension} of $R$ if $\epsilon$ 
is $\natural$-phantom. 
\end{defn}

\begin{disc}
Since the choice of projective resolution above is not canonical, we
must demonstrate that whether $\epsilon\in\Ext_R^1(Q,R)$ is
phantom or not is independent of the choice of $P_\bullet$. Let 
$Q_\bullet$ be another projective resolution of $Q$ consisting
of finitely generated projective modules. 
Continue using $(-)^\vee$ to denote $\Hom_R(-,R)$.

Given $\epsilon$ in
$\Ext_R^1(Q,R)$ representing the short exact sequence 
$$0\ra R\overset\alpha\ra M\ra Q\ra 0,$$ 
let $\phi\in P_1^\vee$ and $\phi'\in Q_1^\vee$ be corresponding cocycles. 
Given 
$\phi\in\im(P_0^\vee\ra P_1^\vee)^\natural_{P_1^\vee}$, 
we will show that $\phi'\in\im(Q_0^\vee\ra Q_1^\vee)^\natural_{Q_1^\vee}$
as well.

We can lift the identity map $Q\overset\id\ra Q$ to a map of complexes
as follows:
$$\xymatrix{\cdots\ar[r] & P_1\ar[r] & P_0\ar[r] & Q\ar[r] & 0 \\
\cdots\ar[r] & Q_1\ar[r]\ar[u]_f & Q_0\ar[r]\ar[u] & 
Q\ar[r]\ar[u]_\id & 0} \leqno(\#)
$$
Applying $(-)^\vee$ to $(\#)$ yields the commutative diagram
$$\xymatrix{\cdots & P_1^\vee\ar[l]\ar[d]^{f^\vee} & P_0^\vee\ar[l]\ar[d] 
& Q^\vee\ar[l]\ar[d]^\id & 0\ar[l] \\
\cdots & Q_1^\vee\ar[l] & Q_0^\vee\ar[l] & Q^\vee\ar[l] & 0\ar[l]} 
\leqno(\#^\vee)
$$
Via the Yoneda correspondence, $\phi\mapsto\phi'$ under the map $f^\vee$.
Since the element $\phi$ is in 
$\im(P_0^\vee\ra P_1^\vee)_{P_1^\vee}^\natural$, applying
Axiom~(4) to $f^\vee$ gives us 
$$\phi'=f^\vee(\phi)\in \im(P_0^\vee\ra Q_1^\vee)_{Q_1^\vee}^\natural.$$ 
Since
$\im(P_0^\vee\ra Q_1^\vee)\subseteq\im(Q_0^\vee\ra Q_1^\vee)$,
Axiom~(3) shows that $\phi'$ is in 
$\im(Q_0^\vee\ra Q_1^\vee)_{Q_1^\vee}^\natural$,
as claimed.
\end{disc}

\begin{disc}\label{homconstr}
Now that we have a well-defined notion of $\natural$-phantom extensions,
we demonstrate more explicitly what it means
for a module to be \mbox{$\natural$-phantom}. Let $(R,\mathfrak{m})$ be
a complete local domain with a closure operation satisfying the axioms. 
For a finitely generated $R$-module $M$ and
an injection $R\overset\alpha\ra M$, if we set $Q=M/\alpha(R)$,
we have the short exact sequence 
$$
0\ra R\overset\alpha\ra M\ra Q\ra 0. \leqno\deqno\label{SES}
$$
Let $w_1=\alpha(1)$, and let $w_2,\ldots,w_{n}$ be elements of $M$ such that
the images $\overline{w_2},\ldots,\overline{w_{n}}$ in $Q$ form a
minimal generating set for $Q$.
Then $w_1,\ldots,w_n$ generate $M$. Let 
$$
R^m\overset{\nu}\ra R^{n-1}\overset{\mu}\ra Q\ra 0 
\leqno\deqno\label{Qres}
$$
be a minimal free presentation of $Q$, 
where $R^{n-1}$ has 
basis $f_2,\ldots,f_n$ such that $\mu$ is
given by $f_i\mapsto \overline{w_i}$. We can also choose a basis
for $R^m$ such that $\nu$ is given by the $(n-1)\times m$ matrix
$$
\nu:= \left(
\begin{matrix}
b_{21} & \cdots & b_{2m} \\
\vdots & \ddots & \vdots \\
b_{n1} & \cdots & b_{nm}
\end{matrix}\right)
$$
where the entries $b_{ij}$ are in $\mathfrak{m}$. 
We can then construct the diagram
$$
\xymatrix{ 
R^m\ar[r]^{\nu_1}\ar[d]^{\id} & R^n\ar[r]^{\mu_1}\ar[d]^{\pi} 
  & M\ar[r]\ar@{->>}[d] & 0 \\
R^m\ar[r]^{\nu} & R^{n-1}\ar[r]^{\mu} & Q\ar[r] & 0 } 
\leqno\deqno\label{Mres}
$$
where $R^n$ has basis $\epsilon_1,\ldots,\epsilon_n$,
$\pi(\epsilon_i) = f_i$ for $i>1$, and $\pi(\epsilon_1)=0$.
The map $\mu_1$ is given by $\epsilon_i\mapsto w_i$, and
$\nu_1$ is given by the $n\times m$ matrix 
$$
\nu_1:=
\begin{pmatrix}
b_{11} & \cdots & b_{1m} \\
b_{21} & \cdots & b_{2m} \\
\vdots & \ddots & \vdots \\
b_{n1} & \cdots & b_{nm}
\end{pmatrix}
=
\begin{pmatrix}
b_{11} & \cdots & b_{1m} \\
& \nu &
\end{pmatrix},
$$
where $b_{1j}w_1 + b_{2j}w_2+\cdots + b_{nj}w_{n}=0$ in $M$, 
for $1\leq j\leq m$. (Such $b_{1j}$ exist because 
$b_{2j}\overline{w_2}+\cdots + b_{nj}\overline{w_{n}}=0$ in 
$Q=M/Rw_1$.) 

From the construction, it is clear that $(\ref{Mres})$ commutes and that
$\mu_1\circ\nu_1$ is the zero map. The choice of the
$w_i$ implies that $(\ref{Mres})$ is exact at $M$. To see that $\ker \mu_1
\subseteq \im\ \nu_1$, suppose that $r_1w_1 +\cdots r_nw_n =0$ in $M$.
Then $r_2\overline{w_2}+\cdots +r_{n}\overline{w_{n}} = 0$ in $Q$, and so
there exist $s_1,\ldots, s_m$ in $R$ such that 
$$
\nu[(s_1,\ldots,s_m)^{\tr}]
= (r_2,\ldots,r_{n})^{\tr},
$$
where $(-)^{\tr}$ denotes the transpose of a matrix.
Then
$$
\nu_1[(s_1,\ldots,s_m)^{\tr}]
= (r,r_2,\ldots,r_{n})^{\tr}\in R^n.
$$
Since $\mu_1\circ\nu_1 = 0$, we see that
$rw_1+r_2w_2+\cdots+r_{n}w_{n} = 0$ in $M$. Therefore 
$rw_1 = r_1w_1$, and so 
$$
\alpha(r-r_1) = (r-r_1)\alpha(1) = (r-r_1)w_1=0.
$$
Since $\alpha$ is
injective, $r=r_1$ so that the vector $(r_1,\ldots,r_n)^{\tr}$
is in $\im\ \nu_1$. 

We can now conclude that the top row of $(\ref{Mres})$ is a finite 
free presentation of $M$.
(Since we do not know \textit{a priori} whether $w_1,\ldots,w_n$ form a minimal
basis for $M$, we cannot say whether
our presentation of $M$ is minimal.) 

We also obtain a commutative diagram with exact rows:
$$
\xymatrix{ 0\ar[r] & R\ar[r]^\alpha & M\ar[r] & Q\ar[r] & 0 \\
F\ar[r]\ar[u] & R^m\ar[r]^{\nu}\ar[u]_\phi & 
R^{n-1}\ar[r]^{\mu}\ar[u]_\psi
  & Q\ar[r]\ar[u]_{\id} & 0 } \leqno\deqno\label{lift} 
$$
where $F$ is free, $\psi(f_2) = w_2,\ldots, \psi(f_n)=w_n$, 
and $\phi$ is given by the $1\times m$ matrix
$(-b_{11}, \ldots, -b_{1m})$. Because $\alpha$ is injective and 
$b_{1j}w_1+\cdots+b_{nj}w_n = 0$
in $M$ for $1\leq j\leq m$, it is clear that (\ref{lift}) commutes
as claimed. We can then take the dual into $R$ of (\ref{lift}):
$$
\xymatrix{ 0\ar[d] & R\ar[l]\ar[d]^{\phi^{\tr}} & \Hom_R(M,R)\ar[l]\ar[d]
  & \Hom_R(Q,R)\ar[l]\ar[d]^{\id} & 0\ar[l] \\
F & R^m\ar[l] & R^{n-1}\ar[l]_{\nu^{\tr}}
  & \Hom_R(Q,R)\ar[l] & 0\ar[l] } \leqno\deqno\label{liftdual} 
$$
Let $Z$ be the kernel of
$R^m\isom\Hom_R(R^m,R)\ra \Hom_R(F,R)\isom F$,
and let $B$ be the image of $\nu^{\tr}$. Then an element of
$\Ext_R^1(Q,R)$ is an element of $Z/B$. As described in 
\cite[Discussion 5.5]{HH94},
the element of $\Ext_R^1(Q,R)$ corresponding to the short exact sequence 
(\ref{SES})
is represented by the map $\phi:R^m\ra R$ in (\ref{lift}). Equivalently, it
is represented by the image of $\phi^{\tr}$ in (\ref{liftdual}). 
\end{disc}

We can now state an equivalent condition
for a finitely generated $R$-module $M$ to be $\natural$-phantom.

\begin{lemma} \label{phcomp}
Let $(R,\mathfrak{m})$ be a complete local domain
possessing a closure operation satisfying Axioms~(1)--(5). 
Let $M$ be a finitely 
generated module, and let $\alpha:R\ra M$ be an injective
map. Using the notation of Discussion \ref{homconstr}, $\alpha$ is a 
$\natural$-phantom
extension of $R$ if and only if the vector $(b_{11},\ldots, b_{1m})^{\tr}$
is in $B_{R^m}^\natural$, where $B$ is the $R$-span in $R^m$ of the
vectors $(b_{i1}, \ldots, b_{im})^{\tr}$ for $2\leq i\leq n$, 
the $b_{ij}$ are in $\mathfrak{m}$ for $2\leq i\leq n$ and all $j$, and
$(-)^{\tr}$ denotes transpose. 
\end{lemma}
\begin{proof} By our definition and the constructions above, 
$\alpha$ is $\natural$-phantom if and only if the cocycle representing
the corresponding element $\epsilon$ 
in $\Ext_R^1(Q,R)$ is in $(\im\ \nu^{\tr})_{R^m}^\natural$. 
Recall that $\epsilon$ is represented by the image of $\phi^{\tr}$, 
which
is $(-b_{11},\ldots, -b_{1m})^{\tr}$. Moreover, the image of $\nu^{\tr}$ is
the $R$-span of the row vectors of $\nu$, which is the $R$-span
of $(b_{i1}, \ldots, b_{im})^{\tr}$ for $2\leq i\leq n$.
The claim concerning the $b_{ij}$ also follows from the preceding discussion. 
\end{proof}

We can now prove the following fact, an analogue of
\cite[Proposition 5.14]{HH94}, by including Axiom~(6).

\begin{lemma} \label{notbad}
Let $(R,\mathfrak{m})$ be a complete local domain
possessing a closure operation satisfying Axioms~(1)--(6).
Let $M$ be a finitely generated 
$R$-module. If $\alpha:R\ra M$ is a $\natural$-phantom extension, 
then $\alpha(1)\not\in \mathfrak{m} M$. 
\end{lemma}
\begin{proof} Suppose that $\alpha(1)\in\mathfrak{m} M$.
Using the notation from Discussion~\ref{homconstr}, we have $\alpha(1)=w_1$. 
So, $\alpha(1)\in \mathfrak{m} M$ if and only if
$w_1=r_2w_2+\cdots+r_{n}w_{n}$ such that the $r_i$ are in $\mathfrak{m}$. 
This occurs if and only if the vector $(1,-r_2,\ldots,-r_{n})^{\tr}$ is in 
$\im\ \nu_1$. In order for a vector
with first component a unit to be in 
$\im\ \nu_1$, the first row of $\nu_1$
must generate the unit ideal, \emph{i.e.}, 
$(b_{11},\ldots,b_{1m})R = R$. Therefore,
there exists $j_0$ such that $b_{1j_0}\in R\setminus \mathfrak{m}$. 
By Lemma \ref{phcomp},
since $\alpha$ is $\natural$-phantom, the vector 
$(b_{11},\ldots, b_{1m})^{\tr}$ is in
$B_{R^m}^\natural$, where $B$ is the $R$-span in $R^m$ of the
vectors $(b_{i1}, \ldots, b_{im})^{\tr}$ for $2\leq i\leq n$. 
Since $b_{ij}\in \mathfrak{m}$ for $2\leq i\leq n$ and $1\leq j\leq m$,
Axiom~(3) implies that $(b_{11},\ldots, b_{1m})^{\tr}$ is 
in $(\mathfrak{m} R^m)_{R^m}^\natural$. 
Using Axiom~(4), for the projection $R^m\ra R$ mapping onto the 
$j_0^{th}$-coordinate, we see that $b_{1j_0}\in (\mathfrak{m} R)_R^\natural$, 
but $(\mathfrak{m} R)_R^\natural =  \mathfrak{m}$ by Axiom~(6). 
This implies that $b_{1j_0}$
cannot be a unit, and so $\alpha(1)\not\in  \mathfrak{m} M$. 
\end{proof}

\section{The Axioms Induce Balanced Big Cohen-Macaulay Modules}

In order to use $\natural$-phantom extensions and our axioms to produce
balanced big Cohen-Macaulay modules, we will use \textit{module modifications}
as found in \cite{Ho75} and \cite{HH94}. We will leave it to the 
reader to review the details
of the construction in \cite[Discussion 5.15]{HH94}, but we will 
outline the idea and its
connection to $\natural$-phantom extensions.

\begin{disc}\label{bigCMconstr}
We can produce a balanced big Cohen-Macaulay module by starting with $R$ and
successively modifying it to trivialize any relations on systems of 
parameters. Starting
with $M_0 = R$ and $w_0 = 1$, construct a sequence of modules $M_t$ containing
elements $w_t$ and homomorphisms $M_t\to M_{t+1}$ such that 
$w_t\mapsto w_{t+1}$.
Given a module $M_t$, we construct $M_{t+1}$ by selecting a relation 
$x_{k+1}u = x_1u_1 + \cdots + x_ku_k$ in $M_t$, where $x_1, \ldots,x_{k+1}$ 
forms
a partial system of parameters in $R$. Then define 
$$
M_{t+1} := \frac{M_t\oplus Rf_1\oplus \cdots \oplus Rf_k}
{R(-u\oplus x_1f_1\oplus\cdots\oplus
x_kf_k)}
$$
in order to trivialize the relation in $M_{t+1}$. The construction yields 
a natural map
$M_t\to M_{t+1}$, where we define $w_{t+1}$ to be the image of $w_t$. Since
we started the chain with $R$ and $1$, we also have maps $R\to M_t$ with
$1\mapsto w_t$ for all $t$.

Continuing in this
manner, we will define the module $B$ to be the direct limit of all such 
$M_t$. 
The module $B$ has the property that all relations on parameters in $B$ 
are trivial,
but it is \textit{a priori} unclear whether or not we have trivialized 
too much and 
caused $\mathfrak{m}B = B$. To avoid this possibility, we will show that
the image of $1$ from $R$ in $B$ is not in $\mathfrak{m}B$ by showing that
$w_t\not\in \mathfrak{m}M_t$ for all $t$. 

In order to accomplish this goal, we will demonstrate that each
module modification $R\to M_t$ is a $\natural$-phantom extension.
By Lemma \ref{notbad}, we then see that $w_t\not\in \mathfrak{m}M_t$ 
for all $t$,
proving that $B$ is a balanced big Cohen-Macaulay module.
\end{disc}

\begin{disc}\label{natphhom}
Let $R$ be a complete local domain. Let 
$\alpha: R \to M$
be an injective map to a finitely generated $R$-module $M$ with a relation
$$
x_1u_1 + \cdots + x_{k+1}u_{k+1}, \leqno\deqno\label{soprel}
$$ 
where $x_1,\ldots,x_{k+1}$ is part of a system of parameters for $R$.
(We will show in Lemma~\ref{injmod} that $R$ injects into all module
modifications and use the identity map on $R$ as a base case.)

Without
loss of generality, denote the (not necessarily minimal) generators
of $M$ by
$$
m_1 =: \alpha(1), m_2, \ldots, m_{n-k-1}, m_{n-k}:=u_1,\ldots,m_n:=u_{k+1}.
$$
We then have a short exact sequence as in (\ref{SES}):
$$
0\to R\overset\alpha\to M\to Q\to 0,
$$
where $Q := M/\alpha(R) = M/Rm_1$. Using the matrices worked out
for the diagram (\ref{Mres}), we obtain
$$
\xymatrix{
& & 0 & 0 & \\
0 \ar[r] & R\ar[r]^{\alpha} & M\ar[r]\ar[u] & Q\ar[r]\ar[u] & 0 \\
& & R^n\ar[u]^{\mu_1}\ar[r] & R^{n-1}\ar[u]_\mu & \\
& & R^m\ar[u]^{\nu_1}\ar[r] & R^m\ar[u]_\nu
}
\leqno\deqno\label{bigpres}
$$
where $\mu$ and $\mu_1$ are as given in Discussion \ref{homconstr}, 
$$
\nu :=
\begin{pmatrix}
b_{21} & \cdots & b_{2,m-1} & 0 \\
\vdots & \ddots & \vdots & \vdots \\
b_{n-k-1,1} & \cdots & b_{n-k-1,m-1} & 0 \\
b_{n-k,1} & \cdots & b_{n-k,m-1} & x_1 \\
\vdots & \ddots & \vdots & \vdots \\
b_{n1} & \cdots & b_{n,m-1} & x_{k+1}
\end{pmatrix}
$$
(as we can include the relation (\ref{soprel}) without loss of
generality), and 
$$
\nu_1 := 
\begin{pmatrix}
b_{11} & \cdots & b_{1,m-1} & 0 \\
& \hspace{7mm}\nu
\end{pmatrix}.
$$
Then (\ref{bigpres}) gives a commutative diagram with exact rows
and columns and free presentations of $M$ and $Q$. 
\end{disc}

\begin{note} \label{natphhomnote}
Using the construction in Discussion \ref{natphhom}, we define the following:
$$
\begin{array}{l}
\bx := (b_{11}, \ldots,  b_{1,m-1}, 0)^{\tr},
\mbox{ the transpose of the first row of } \nu_1. \\
\by := ( b_{n1}, \ldots, b_{n,m-1}, x_{k+1})^{\tr},
\mbox{ the transpose of the last row of } \nu. \\
H := \mbox{$R$-module generated by all the rows of $\nu$ 
except the last row}. \\
I := (x_1,\ldots, x_k). 
\end{array}
$$
\end{note}

Adapting the characterization of $\natural$-phantom
extensions using free presentations from Lemma \ref{phcomp}, we have

\begin{lemma}\label{alphaph}
Let $R$ be a complete local domain possessing a closure operation
satisfying Axioms~(1)--(5). Use the notation from Discussion~\ref{natphhom}
and Notation~\ref{natphhomnote}.
The map $\alpha$ is $\natural$-phantom if and only if
$$
\bx \in (R\by+ H)^\natural_{R^m}.
$$
\end{lemma}

\begin{disc}\label{modmod}
We now investigate module modifications. Using the notation
developed in Discussion \ref{natphhom}, use free generators
$f_1,\ldots f_k$ to define the modification
$$
M' := \frac{M\oplus Rf_1 \oplus\cdots\oplus Rf_k}
{R(m_n \oplus x_1f_1\oplus\cdots\oplus x_kf_k)}
$$
of $M$ with respect to the relation (\ref{soprel}) on
the parameters $x_1,\ldots,x_{k+1}$. 
Define 
$$
\alpha': R\to M'
$$
as the composition $\alpha' = \beta \circ \alpha$, where
$\beta(u)$ is the image of $u\oplus 0\cdots\oplus 0$ in $M'$ for all $u\in M$.

As stated in Discussion \ref{bigCMconstr}, in order to show that
our sequence of modifications leads to a balanced big Cohen-Macaulay 
module $B$, it suffices to show that $\alpha'$ is $\natural$-phantom
when $\alpha$ is $\natural$-phantom. 
\end{disc}

\begin{lemma} \label{injmod} Using the constructions given in Discussions
\ref{natphhom} and \ref{modmod}, the map $\alpha'$ is injective.
\end{lemma}
\begin{proof}
Since $\alpha' = \beta\circ\alpha$ and $\alpha$ is injective, it suffices
to show
that $\beta$ is also injective. If $\beta(u) = 0$ in $M'$, then 
$$
u\oplus 0\oplus\cdots\oplus 0 = r(m_n\oplus x_1f_1\oplus\cdots\oplus x_kf_k)
$$
so that $rx_1=0$. As $R$ is a domain, $r=0$ and thus $u=r m_n = 0$.
\end{proof}

One can construct free presentations for $M'$ and 
$Q':=M'/\alpha'(1)$ similarly 
to how we constructed free presentations for $M$ and $Q$ in 
(\ref{bigpres}) because
we added $k$ new generators and one new relation.

\begin{lemma} The following diagram is commutative and 
gives free presentations for $M'$ and $Q'$:
$$
\xymatrix{
& & 0 & 0 & \\
0 \ar[r] & R\ar[r]^{\alpha'} & M'\ar[r]\ar[u] & Q'\ar[r]\ar[u] & 0 \\
& & R^{n+k}\ar[u]^{\mu'_1}\ar[r] & R^{n-1+k}\ar[u]_{\mu'} & \\
& & R^{m+1}\ar[u]^{\nu'_1}\ar[r] & R^{m+1}\ar[u]_{\nu'}
}
\leqno\deqno\label{bigpres'}
$$
where the matrices
$\mu'$ and $\mu'_1$ are extensions of $\mu$ and $\mu_1$ that 
account for the $k$
new generators $w_1,\ldots,w_k$, $\nu'$ is the $(n-1+k)\times(m+1)$ matrix
$$
\nu'  :=  
\begin{pmatrix}
 &  \vline & 0 \\
 & \vline & \vdots \\
 \nu &  \vline & 0 \\
 &  \vline & 1 \\
 \hline
 & \vline & x_1 \\
\mathbf{0}  & \vline & \vdots \\
 & \vline & x_k \\
 \end{pmatrix},
$$
and $\nu'_1$ is the $(n+k)\times(m+1)$ matrix
$$
\nu'_1 :=
\begin{pmatrix}
b_{11} & \cdots & b_{1,m-1} & 0 & 0\\
& \hspace{7mm}\nu'
\end{pmatrix}.
$$
\end{lemma}
\begin{proof}
All of the claims are straightforward to verify. 
We will describe the exactness 
at $R^{n+k}$ though. Indeed, $\mu'_1\circ \nu'_1 = 0$ as $\mu_1\circ\nu_1 =0$
and $m_n + x_1w_1+\cdots+x_kw_k=0$ in $M'$. To see why the kernel of $\mu'_1$
is contained in the image of $\nu'_1$, suppose that 
$\mu'_1(r_1,\ldots,r_{n+k})^{\tr} = 0$.
Then 
$$
r_1m_1 + \cdots r_nm_n + r_{n+1}w_1 +\cdots+ r_{n+k}w_k = 
r(m_n + x_1w_1+\cdots + x_kw_k)
$$
in $M\oplus R^k$. Thus, 
$$
r_1m_1 + \cdots + r_{n-1}m_{n-1} + (r_n - r)m_n = 0
$$
in $M$ and $r_{n+i} = rx_i$ for all $1\leq i\leq k$. Thus, 
$(r_1,\ldots,r_{n+k})^{\tr}$
is an $R$-linear combination of the columns of $\nu'_1$.
\end{proof}

We now characterize when $\alpha'$ is $\natural$-phantom. The 
proof follows directly from
Lemma~\ref{phcomp} and our computations of $\nu'_1$ and $\nu'$.
 
\begin{lemma}\label{aplha'ph}
Let $R$ be a complete local domain possessing a closure
operation satisfying Axioms~(1)--(5).
Use Notation \ref{natphhomnote}.
The map $\alpha'$ is $\natural$-phantom if and only if
$$
\bx\oplus 0 \in (R(\by\oplus 1) + (H\oplus 0) + 
I(\mathbf{0}\oplus 1))^\natural_{R^{m+1}}.
$$
\end{lemma}

We now apply Axiom~(6) to connect the conditions characterizing
when $\alpha$ and $\alpha'$ are $\natural$-phantom. 

\begin{lemma} 
Let $R$ be a complete local domain possessing a closure operation
satisfying Axioms~(1)--(6).
Use Notation \ref{natphhomnote}.
Let $\bv$ be in $R^m$. If $\bv\in (I\by + H)^\natural_{R^m}$, then
$\bv\oplus 0\in (R(\by\oplus 1) + (H\oplus 0) + 
I(\mathbf{0}\oplus 1))^\natural_{R^{m+1}}$.
\end{lemma}
\begin{proof}
By Axiom~(6), $0^\natural_{R} = 0$. By Lemma \ref{proplist}(b),
$$
(I(\by\oplus 0) + (H\oplus 0) )^\natural_{R^{m+1}}
= (I\by+H)^\natural_{R^m} \oplus 0^\natural_R 
= (I\by+H)^\natural_{R^m} \oplus 0.
$$
Since $I(\by\oplus 0)\subseteq R(\by\oplus 1) + I(\mathbf{0}\oplus 1)$,
Axiom~(3) finishes the proof.
\end{proof}

\begin{prop}\label{natphcond}
Let $R$ be as in the previous lemma.
Use Notation \ref{natphhomnote}.
If $\bx \in (I\by + H)^\natural_{R^m}$, then $\alpha'$ is $\natural$-phantom.
\end{prop}

Before we apply Axiom~(7) (the generalized colon-capturing property), 
we look again at the closure conditions found in Lemma \ref{alphaph} and
Proposition \ref{natphcond}.

\begin{lemma}
Let $R$ be a complete local domain possessing a closure operation
satisfying Axioms~(1)--(5).
Use Notation \ref{natphhomnote}, let $Q:=R^m/H$, let $v:=\by + H$, and let
$\ba\in R^m$. 
\begin{alist}
\item $\ba \in (R\by+H)^\natural_{R^m}$ if and only if 
   $\ba+H\in(Rv)^\natural_{Q}$.
\item $\ba \in (I\by +H)^\natural_{R^m}$ if and only if 
   $\ba+H\in(Iv)^\natural_{Q}$.
\end{alist}
\end{lemma}

We can now achieve the goal we set at the end of Discussion
\ref{modmod}.

\begin{prop} Let $R$ be a complete local domain possessing
a closure operation satisfying Axioms~(1)--(7).
Use Notation \ref{natphhomnote}. Let $Q:=R^m/H$, 
and let $v:=\by + H$.
If $\alpha$ is $\natural$-phantom, then $\alpha'$ is 
$\natural$-phantom.
\end{prop}
\begin{proof} By Lemma \ref{alphaph} and the last lemma, 
$\alpha$ is $\natural$-phantom if and only if $\bx + H \in
(Rv)^\natural_Q$. 
Let $\pi:R^m \to R/I$ be the surjection that projects 
an element of $R^m$ onto its last coordinate modulo $I$. 
Since the last entries of the generators of $H$ exactly generate
$I=(x_1,\ldots,x_k)$, we see that $\pi(H) = I$, and so
$\pi$ factors through $Q$. Let $f:Q\to R/I$ be the resulting
surjection. As the last entry of $\bx$ is 0, 
we have that $\bx+H$ is also in the kernel of $f$. By Axiom~(7), 
$\bx +H \in (Iv)^\natural_Q$ and so $\bx \in (I\by +H)^\natural_{R^m}$,
which implies that $\alpha'$ is $\natural$-phantom by 
Proposition~\ref{natphcond}.
\end{proof}

Since $\alpha:R\to M$ being $\natural$-phantom implies that
the map $\alpha':R\to M'$ is $\natural$-phantom, where $M'$ is
a module modification of $M$, we can conclude the following
theorem by using Lemma~\ref{notbad} and 
Discussions \ref{bigCMconstr} and \ref{modmod}.

\begin{thm}
Let $R$ be a complete local domain possessing a closure operation 
that satisfies Axioms~(1)--(7). Then
there exists a balanced big Cohen-Macaulay $R$-module $B$.
\end{thm}

\section{Big Cohen-Macaulay Modules Imply the Axioms}

In this section, assume that $(R,\mathfrak{m})$ is a complete 
local domain and that
$B$ is a balanced big Cohen-Macaulay $R$-module. We will define a
closure operation for $R$-modules, based on $B$, that will satisfy
Axioms \ref{theaxioms}.

\begin{defn}
Given $R$ and $B$ as above and finitely generated $R$-modules
$N\subseteq M$, define $N^\natural_M$ to be the set of all $u\in M$
such that for all $b\in B$ we have 
$b\tensor u\in\im(B\tensor N\to B\tensor M)$.
\end{defn}

\begin{thm}
The closure operation $N^\natural_M$
defined above in terms of the balanced big Cohen-Macaulay
$R$-module $B$ satisfies Axioms \ref{theaxioms}.
\end{thm}
~
\begin{proof}
~
\begin{nlist}

\item The claim that $N\subseteq N^\natural_M$ is clear 
and basic tensor product
properties imply that $N^\natural_M$ is a submodule of $M$.

\item By Axiom~(1), $N^\natural_M\subseteq (N^\natural_M)^\natural_M$. Suppose
that $u\in (N^\natural_M)^\natural_M$, and let $b\in B$. Then 
$b\tensor u \in \im(B\tensor N^\natural_M\to B\tensor M)$, which implies that
$b\tensor u = \sum_j b_j\tensor v_j$ in $B\tensor M$, where $b_j\in B$ and
$v_j \in N^\natural_M$. Since each $b_j\tensor v_j \in 
\im(B\tensor N\to B\tensor M)$,
we have $b\tensor u \in \im(B\tensor N\to B\tensor M)$. Thus, 
$u\in N^\natural_M$.

\item Given $N\subseteq M\subseteq W$, the inclusion $N^\natural_W \subseteq 
M^\natural_W$ follows from the fact that 
$\im(B\tensor N\to B\tensor W)\subseteq \im(B\tensor M\to B\tensor W)$.

\item Given $N\subseteq M$ and $f:M\to W$, the inclusion
$f(N^\natural_M) \subseteq f(N)^\natural_W$ follows from the fact that
the map $B\tensor f$ maps $\im(B\tensor N\to B\tensor M)$ to
$\im(B\tensor f(N)\to B\tensor W)$.

\item Suppose that $N^\natural_M=N$ and that 
$\overline{u} \in 0^\natural_{M/N}$.
Then for all $b\in B$, we have $b\tensor \overline{u} = 0$ in 
$B\tensor M/N \isom (B\tensor M)/\im(B\tensor N\to B\tensor M)$. Thus, 
$b\tensor u \in \im(B\tensor N\to B\tensor M)$ so that $u\in N^\natural_M=N$.
Therefore, $\overline{u} = 0$ in $M/N$.

\item Suppose that $x\in \mathfrak{m}^\natural_R \setminus \mathfrak{m}$.
Then $x$ is a unit in $R$, and for all $b\in B$ we have $xb\in \mathfrak{m}B$.
Therefore, $x$ annihilates $B/\mathfrak{m}B$, but since $x$ is a unit, 
$B = \mathfrak{m}B$, contradicting the fact that $B$ is a balanced big
Cohen-Macaulay $R$-module. 

Now, suppose that $x\in 0^\natural_R$. Then $xb=0$ in $B$ for all $b\in B$.
Since $R$ is a domain and $B$ is a balanced big Cohen-Macaulay $R$-module,
$x$ cannot annihilate nonzero elements of $B$ unless $x=0$.

\item Suppose $f:M\to R/I$ (we do not need to assume it is surjective), 
where $x_1,\ldots,x_{k+1}$
is a partial system of parameters for $R$, $I = (x_1,\ldots,x_k)$, and 
$f(v)=x_{k+1}+I$. 
Let $u\in (Rv)^\natural_M \cap \textup{ker}\, f$, and let $b\in B$. Then
$f(u) = 0$, and $b\tensor u = c_b \tensor v$ in $B\tensor M$. Therefore,
$c_b \tensor f(v) = 0$ in $B\tensor R/I$, which implies that
$c_b x_{k+1} + IB = IB$ in $B/IB$. Hence, $c_b x_{k+1} \in IB$. Since $B$
is balanced big Cohen-Macaulay, $c_b \in IB$. Thus, 
$$
b\tensor u \in \im(IB\tensor Rv\to B\tensor M)
= \im(B\tensor Iv\to B\tensor M),
$$
and so $u\in (Iv)^\natural_M$.

\end{nlist}~
\end{proof}

\section{Common Closure Operations and the Axioms}

In the final section we compare tight closure, plus closure, Frobenius
closure, and solid closure with the axioms. 
We refer the reader to 
\cite{HH90}, especially Section~8, for the definitions and notation involved
in tight closure of modules.
See \cite{HH92} and \cite{Sm94} for the details on
plus closure and its properties. See \cite[Section 7]{HH94} for information
on Frobenius closure, and see \cite{Ho94} for details
on solid closure.

It is easy to see that any closure operation that can be defined in terms
of tensor products with $R$-modules or $R$-algebras, e.g.,
plus closure, Frobenius closure, or solid closure, will
satisfy Axioms~(1)--(5). These three closure operations
also satisfy Axiom~(6) as each is contained
in the integral closure. The crucial axiom to verify is thus Axiom~(7). 

\begin{exam}
Plus closure
in prime characteristic $p$ satisfies Axiom~(7) as it is defined 
in terms of the balanced big Cohen-Macaulay $R$-algebra $R^+$, the
absolute integral closure of $R$. 
\end{exam}

\begin{exam}
Frobenius closure does not satisfy Axiom~(7) in general as it does not 
even satisfy
colon-capturing. For example, let $K$ be a field of prime characteristic
$p \equiv 1 \pmod{3}$. Set $R = K[X,Y,Z]/(X^3+Y^3+Z^3) = K[x,y,z]$, 
the cubical cone, and set $S = K[s,t]$. Then let $T = K[xs,ys,zs,xt,yt,zt]$
be the Segre product of $R$ and $S$. After localization at the maximal ideal
and completion, $\wh{T}$ is a complete local domain that is not 
Cohen-Macaulay. It can be checked that $ys$, $xt$, and $xs-yt$ form
a system of parameters, but the relation 
$(zs)(zt)(xs - yt) = (zs)^2(xt) - (zt)^2(ys)$ shows that we do not
have a regular sequence. (See \cite[pp. 15--16]{Ho07} for the details.)
On the other hand, $R$ is $F$-pure (see \cite[pp. 162, 269]{Ho07})
and so is $R[s,t]$. Since $T$ is a direct summand of $R[s,t]$ as
a $T$-module, $T$ and $\wh{T}$ are also $F$-pure. Therefore,
$I^F= I$ for all ideals in $\wh{T}$ and so colon-capturing
cannot hold in $\wh{T}$ since it is not Cohen-Macaulay.

\end{exam}

\begin{exam}
In positive characteristic, solid
closure and tight closure coincide for complete local domains 
\cite[Theorem 8.6]{Ho94}.
Proposition \ref{TCAxiom7} below will then show that solid closure satisfies
Axiom~(7) in this case. In equal characteristic 0, solid closure
contains the various notions of tight closure in equal characteristic 0
\cite[Theorem 11.4]{Ho94}. Since those tight closure operations capture
colons, solid closure does as well. Whether or not solid closure obeys
Axiom~(7) in characteristic 0 is unknown.
In mixed characteristic, the situation is more mysterious.
It is not even known whether
a complete local domain having all parameter ideals solidly closed
must be Cohen-Macaulay for dimension greater than two, 
and so colon-capturing may not hold.
\end{exam}

\begin{exam}
In positive characteristic, tight closure
satisfies all of Axioms
\ref{theaxioms} for  a complete local domain $R$. 
Axioms~(1), (2), and (3) are given in \cite[Proposition
8.5(a), (e), (b)]{HH90}. Axiom~(5) is found in \cite[Remark 8.4]{HH90}.
Axiom~(6) is a consequence of \cite[Theorem 5.2]{HH90}
since $0$ and $\mathfrak{m}$ are integrally closed in $R$.  

For lack of a convenient reference, we will outline the proof for Axiom~(4) 
before proving Axiom~(7), the
generalized colon-capturing property.
\end{exam}

\begin{lemma} Let $R$ be any ring of prime characteristic $p>0$. Then
Axiom~\textnormal{(4)} holds, \emph{i.e.}, 
given $f:M\to W$ with $N\subseteq M$, we have
$f(N^*_M) \subseteq f(N)^*_W$.
\end{lemma}
\begin{proof}
If $u\in N^*_M$, then there exists a $c\in R$ not in 
any minimal prime such that
$c\otimes u \in \im(\Fe(N)\to \Fe(M))$. Thus, 
$c\otimes f(u) \in \im(\Fe(f(N))\to \Fe(W))$,
so $f(u)\in f(N)^*_W$.
\end{proof}

\begin{prop}\label{TCAxiom7}
Let $R$ be a complete local domain of prime characteristic $p>0$. 
Axiom~\textnormal{(7)}, the generalized colon-capturing property, 
holds for tight closure.
\end{prop}
\begin{proof}
Let $f:M\to R/I$ be a homomorphism, where $x_1,\ldots,x_{k+1}$
is a partial system of parameters for $R$, $I = (x_1,\ldots,x_k)$,
and $f(v)=x_{k+1}$. 
Suppose that $u\in (Rv)^*_M \cap \textup{ker}\, f$. We need
to show that $u\in (Iv)^*_M$. Let $c$ be a test
element for $R$, which exists by \cite[Corollary 6.24]{HH90}.
Then for all $q=p^e$, we have
$cu^q \in (Rv)^{[q]}_M$. Applying $f$, we have
$cf(u)^q = r_qx_{k+1}^q$ in $R/I^{[q]}$. Since $f(u)=0$, 
$r_qx_{k+1}^q \in I^{[q]}$ so that $r_q \in I^{[q]}: x_{k+1}^q$.
By regular colon-capturing for tight closure \cite[Theorem~4.7]{HH90}, 
$r_q \in (I^{[q]})^*$ and so $cr_q\in I^{[q]}$ since $c$ is a test element.
Therefore, $c^2u^q = cr_qx_{k+1}^q\in (Iv)^{[q]}_M$ and so
$u\in (Iv)^*_M$.
\end{proof}

\begin{exam}
Using the usual reduction to characteristic $p$ arguments
(see \cite{TCZ}, specifically Theorem 4.1.7 on colon-capturing),
one can deduce that equal characteristic 0 tight closure notions
also obey Axioms~(1)--(7).
\end{exam}

\section*{Acknowledgments}

My thanks to Mel Hochster for several helpful conversations during
the preparation of this manuscript and for pointing out the example
for an $F$-pure but non-Cohen-Macaulay complete local domain.
I also thank the anonymous referee for many helpful suggestions that
improved the exposition of the paper.


\begin{thebibliography}{HHH}

\bibitem[D]{D} \textsc{G. Dietz}, \textit{Big Cohen-Macaulay algebras
and seeds}, Trans. Amer. Math. Soc. \textbf{359} (2007), no. 12, 
5959--5989. MR2336312 (2008h:13021)  

\bibitem[Ho1]{Ho75} \textsc{M. Hochster}, \textit{Topics in the
homological theory of modules over commutative rings},~CBMS 
Regional Conf.\ Ser.\ in Math.\ \textbf{24}, Amer.\ Math.\ Soc., Providence,
RI, 1975. MR0371879  (51:8096)  

\bibitem[Ho2]{Ho94} \textsc{M. Hochster}, \textit{Solid closure}, in:
\textit{Commutative Algebra: Syzygies, Multiplicities, and Birational
Algebra}, Contemp. 
Math. \textbf{159}, Amer. Math. Soc., Providence, RI, 1994, 103--172. 
MR1266182 (95a:13011)  

\bibitem[Ho3]{Ho07} \textsc{M. Hochster}, 
\textit{Foundations of tight closure theory}, lecture notes available at
\texttt{http://www.math.lsa.umich.edu/$\sim$hochster/mse.html}.

\bibitem[HH1]{HH90} \textsc{M. Hochster} and \textsc{C. Huneke}, \textit{Tight 
closure, invariant theory, and the Brian\c{c}on-Skoda theorem}, J. Amer. 
Math. Soc. \textbf{3} (1990), 31--116.
MR1017784 (91g:13010)  

\bibitem[HH2]{HH92} \textsc{M. Hochster} and \textsc{C. Huneke}, 
\textit{Infinite
integral extensions and big Cohen-Macaulay algebras}, Annals of Math.
\textbf{135} (1992), 53--89.
MR1147957 (92m:13023)  

\bibitem[HH3]{HH94} \textsc{M. Hochster} and \textsc{C. Huneke}, \textit{Tight 
closure of parameter ideals and splitting in module-finite  extensions}, J. 
Algebraic Geom.  \textbf{3}  (1994),  no. 4, 599--670.
MR1297848 (95k:13002)  

\bibitem[HH4]{TCZ} \textsc{M. Hochster} and \textsc{C. Huneke}, 
\textit{Tight Closure in Equal Characteristic Zero}, preprint
available at \texttt{http://www.math.lsa.umich.edu/$\sim$hochster/msr.html}.

\bibitem[Sm]{Sm94} \textsc{K.E. Smith}, \textit{Tight closure of parameter
ideals}, Invent. Math. \textbf{115} (1994), 41--60. MR1248078 (94k:13006)

\end{thebibliography}
\end{document}